\documentclass[conference]{IEEEtran}
\IEEEoverridecommandlockouts
\usepackage{cite}
\usepackage{amsmath,amssymb,amsfonts}
\usepackage{algorithmic}
\usepackage{graphicx}
\usepackage{booktabs}
\usepackage{url}
\usepackage{flushend}
\newtheorem{proposition}{Proposition}
\newcounter{algorithm}
\newcommand{\B}{\mathbb{B}}
\newcommand{\R}{\mathbb{R}}
\newcommand{\norm}[1]{\left\lVert#1\right\rVert}
\begin{document}
\title{Time-Efficient Active Bearing-Only Localization with Reception and Coverage Guarantees}
\author{\IEEEauthorblockN{Ao Xiao\textsuperscript{*}, Fangfang Zhou, and Qiteng Guo\thanks{The authors are with the College of Mathematics and Statistics,\protect\newline Northwest Normal University, Lanzhou, China.\protect\newline \textsuperscript{*}Corresponding author: Ao Xiao (\texttt{xiaoao2math@163.com}).}}}
\maketitle

\begin{abstract}
We consider the localization and removal of a stationary omnidirectional radio source using a mobile sensor with bounded bearing errors and an unknown reception radius within known bounds. The objective is to minimize expected remaining action time under a fixed continuation policy while guaranteeing reception and source removal under the sensing model. A three-disk filter ensures reception, and a minimum enclosing circle certifies coverage of all bearing-consistent source positions. Monte Carlo sample-average search ranks a finite candidate set by complete mission time, including travel, readings, terminal service, and a finite fallback scan that ensures completion under the stated assumptions. In the reference case, 5000 independent validation missions average 60.40 s (approximate 95\% confidence interval: [59.96, 60.84] s), saving 19.39\% relative to a prescribed lateral point and 0.79\% relative to a short-travel point. All 25,000 selected-point validation missions across five initial geometries succeed, with lower mean completion times than both prescribed designs. A same-grid objective ablation shows that minimizing travel or reading count alone can increase mean completion time.
\end{abstract}

\section{Introduction}
A single bearing leaves source range uncertain. A lateral move improves triangulation but costs time and may exceed reception range. Selecting the next observation must therefore balance travel, reception, and coverage at the final service point.

Tokekar \emph{et al.} studied active robotic radio telemetry \cite{telemetry}; Vander Hook \emph{et al.} included travel and measurement time in cautious greedy localization \cite{cautious}. Fisher-information analyses characterize favorable statistical estimation geometries \cite{bishop,zhao}. For bounded angular errors, Tokekar and Isler derived worst-case area and diameter guarantees from wedge intersections \cite{bounded}. Pozzan \emph{et al.} combined weighted least squares with covariance-driven active sensing under Gaussian noise \cite{active}. Our reception and coverage criteria differ from those formulations.

We study one stationary omnidirectional source after detection, using Problem B of the 2026 Contemporary Undergraduate Mathematical Contest in Modeling (CUMCM) \cite{benchmark}. A related preprint's abstract describes whole-region search and completion certification for multiple emitters with unknown directionality \cite{cao}. We instead focus on the reception-constrained second observation, treating the unknown reception radius as a bounded physical parameter.

We combine a derived reception-safe region, an enclosing-circle stopping rule, and sample-average selection under fixed continuation. The geometric guarantees require no source distribution; the assumed distribution ranks feasible points by expected completion time. We evaluate this integration in five geometries, cost sensitivity, and a selection-objective ablation.

\section{Sensing Model and Geometric Foundation}
\subsection{Bounded Bearings and Reception}
Let the source location be $G\in\mathcal D=\B(0,a)\subset\R^2$, where $\B(C,r)$ denotes a closed disk. Observation positions $A_i$ are accurately known and may lie outside $\mathcal D$. Motion is unobstructed at constant speed $v_{\mathrm{m}}$. The source is stationary and omnidirectional, with a fixed but unknown reception radius
\begin{equation}
 R_{\min}\leq R\leq R_{\max}.
 \label{eq:range}
\end{equation}
A normal observation at distance $d_i=\norm{G-A_i}$ is available when $r_{\mathrm{near}}<d_i\leq R$. At $d_i\leq r_{\mathrm{near}}$, a close-range flag permits terminal service at $A_i$. Otherwise no signal is received. A normal bearing satisfies
\begin{equation}
 \begin{aligned}
 \theta_i&=\operatorname{atan2}(G_y-A_{iy},G_x-A_{ix})+\epsilon_i,\\
 |\epsilon_i|&\leq\alpha,\qquad 0<\alpha<\pi/2,
 \end{aligned}
 \label{eq:bearing}
\end{equation}
with angles interpreted modulo $2\pi$. The guarantees require only the bound. For numerical expectations, errors at distinct locations are independent uniform variables on $[-\alpha,\alpha]$; repeated locations reuse the original bearing.

\subsection{Wedge Intersection}
Define the local orthonormal basis
\begin{equation}
 e_i=(\cos\theta_i,\sin\theta_i)^T,\qquad
 n_i=(-\sin\theta_i,\cos\theta_i)^T.
\end{equation}
For a candidate source $P$, write $s_i(P)=(P-A_i)^Te_i$ and $t_i(P)=(P-A_i)^Tn_i$. The positions consistent with this reading form the wedge
\begin{equation}
 W_i=\{P:s_i(P)\geq0,\ |t_i(P)|\leq s_i(P)\tan\alpha\}.
 \label{eq:wedge}
\end{equation}
Since $\alpha>0$, the two lateral inequalities imply $s_i\geq0$. Let $\mathcal I_n$ index the retained normal bearings among the first $n$ attempted readings. Then $\Omega_n=\bigcap_{i\in\mathcal I_n}W_i$ is an intersection of $2|\mathcal I_n|$ half-planes. We retain feasible nonparallel boundary intersections, remove duplicates, and check unbounded directions. If $\Omega_n$ is nonempty and bounded, its vertex set $V_n$ satisfies $\Omega_n=\operatorname{conv}(V_n)$ and
\begin{equation}
 D_n=\max_{p,q\in V_n}\norm{p-q}.
 \label{eq:diameter}
\end{equation}
This follows from convexity and the triangle inequality.

\begin{figure}[t]
 \centering\includegraphics[width=0.80\columnwidth]{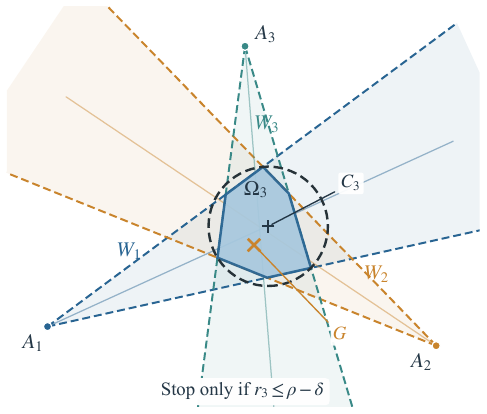}
 \caption{Bearing wedges $W_i$ intersect in $\Omega_3$. The dashed black circle encloses the whole intersection, with center $C_3$ and radius $r_3$; a final approach to its center requires $r_3\leq\rho-\delta$. The display half-angle is $12^\circ$ for visibility, rather than the experimental $1^\circ$.}
 \label{fig:wedges}
\end{figure}

\subsection{A Certified Final Approach}
The terminal operation requires the source to lie within a radius $\rho$ of the arrival point. Define
\begin{equation}
 \begin{aligned}
 r_n&=\min_C\max_{P\in\Omega_n}\norm{P-C},\\
 C_n&\in\arg\min_C\max_{P\in\Omega_n}\norm{P-C}.
 \end{aligned}
 \label{eq:mec}
\end{equation}
For bounded $\Omega_n$, it suffices to enclose $V_n$. A minimum enclosing circle is supported by at most three vertices \cite{welzl}. We test pair-diameter circles and noncollinear triple circumcircles, selecting the smallest that encloses every vertex. A singleton has radius zero; an unbounded region has $r_n=\infty$.

The condition $D_n\leq2\rho$ is necessary for coverage by a disk of radius $\rho$, but is not sufficient. For a nondegenerate acute triangle, the minimum enclosing circle is its circumcircle. If $D$ is the longest side and $\gamma<\pi/2$ its opposite angle, the enclosing radius is $D/(2\sin\gamma)>D/2$. A triangle scaled to $D=2\rho$ therefore passes the diameter test but cannot be covered by a radius-$\rho$ disk.

We instead stop normal sensing when
\begin{equation}
 r_n\leq\rho-\delta,
 \label{eq:stop}
\end{equation}
and move to $C_n$, as illustrated in Fig.~\ref{fig:wedges}. Indeed,
\begin{equation}
 G\in\Omega_n\subseteq\B(C_n,r_n)
 \ \Longrightarrow\ \norm{G-C_n}\leq r_n\leq\rho-\delta.
 \label{eq:certificate}
\end{equation}
We use $\delta=0.01$ m and inflate computed radii for numerical tolerance. Coverage is checked on $\Omega_n$, independently of the planning region; finite completion also requires the fallback in Section IV.

\section{Reliable Reception for the Next Observation}
\subsection{The Exact Conditional Constraint}
The first normal observation gives the possible source set
\begin{equation}
 \mathcal F_1=\mathcal D\cap W_1\cap
 \{G:r_{\mathrm{near}}<\norm{G-A_1}\leq R_{\max}\}.
 \label{eq:prior}
\end{equation}
For a fixed $G\in\mathcal F_1$, the smallest radius compatible with the first reading is $\max\{R_{\min},\norm{G-A_1}\}$. Hence the exact region guaranteeing either a normal signal or a close-range flag is
\begin{equation}
 \mathcal C_{\mathrm{recv}}=
 \bigcap_{G\in\mathcal F_1}
 \B\bigl(G,\max\{R_{\min},\norm{G-A_1}\}\bigr).
 \label{eq:exact}
\end{equation}
The intersection is closed and convex. Close-range feedback counts as usable reception because it permits service at the current point.

\subsection{Three-Disk Reduction}
Enlarge $\mathcal F_1$ to the complete sector
\begin{equation}
 \widetilde{\mathcal F}_1=
 \{A_1+r e(\theta_1+\omega):0\leq r\leq R_{\max},\ |\omega|\leq\alpha\},
\end{equation}
where $e(\theta)=(\cos\theta,\sin\theta)^T$. Removing the source-domain and close-range restrictions makes the filter conservative.

\begin{proposition}
For $R_{\max}\geq R_{\min}>0$ and $0<\alpha<\pi/2$, replacing $\mathcal F_1$ by $\widetilde{\mathcal F}_1$ in \eqref{eq:exact} gives exactly
\begin{equation}
 \begin{split}
 \mathcal C_{\mathrm{safe}}={}&\B(A_1,R_{\min})\\
 &\cap\B(A_1+R_{\min}e(\theta_1-\alpha),R_{\min})\\
 &\cap\B(A_1+R_{\min}e(\theta_1+\alpha),R_{\min}).
 \end{split}
 \label{eq:three}
\end{equation}
Consequently $\mathcal C_{\mathrm{safe}}\subseteq\mathcal C_{\mathrm{recv}}$.
\end{proposition}
\begin{IEEEproof}
Write $A_2=A_1+u e_1+v n_1$ and $h_\omega=u\cos\omega+v\sin\omega$. For a sector point at radius $r$,
\begin{equation}
 \norm{A_2-G}^2=r^2+u^2+v^2-2r h_\omega.
 \label{eq:distance}
\end{equation}
For $0\leq r\leq R_{\min}$, this expression is convex in $r$; its maximum occurs at an endpoint. The endpoints require $u^2+v^2\leq R_{\min}^2$ and $u^2+v^2\leq2R_{\min}h_\omega$. For $r\geq R_{\min}$, the reception inequality reduces to $u^2+v^2\leq2r h_\omega$. The endpoint condition already ensures $h_\omega\geq0$, so larger $r$ cannot tighten it.

Necessity of the two angular-endpoint conditions implies $u\geq0$. For such $u$, $\cos\omega\geq\cos\alpha$ and $v\sin\omega\geq-|v|\sin\alpha$, with equality at one endpoint. Therefore
\begin{equation}
 \min_{|\omega|\leq\alpha}h_\omega
 =u\cos\alpha-|v|\sin\alpha.
\end{equation}
Checking the sector apex and the two points at $r=R_{\min}$, $\omega=\pm\alpha$, is thus necessary and sufficient. These are precisely the three disks in \eqref{eq:three}. Enlarging the source set can only shrink its intersection of reception disks.
\end{IEEEproof}

The implementation uses the equivalent two-inequality test
\begin{equation}
 \begin{aligned}
 u^2+v^2&\leq R_{\min}^2,\\
 u^2+v^2&\leq2R_{\min}(u\cos\alpha-|v|\sin\alpha).
 \end{aligned}
 \label{eq:filter}
\end{equation}
The test is exact for the enlarged sector and conservative for $\mathcal F_1$. Constraints beyond $R_{\min}$ are redundant; $R_{\max}$ affects the source region but not the filter (Fig.~\ref{fig:safe}).

\begin{figure}[t]
 \centering\includegraphics[width=0.80\columnwidth]{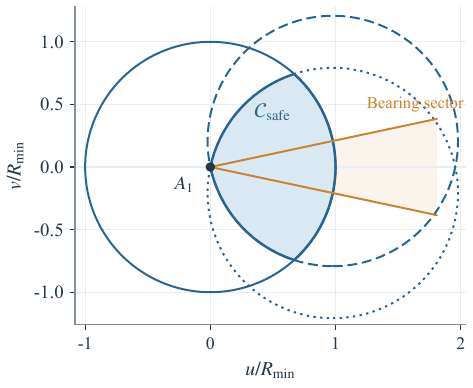}
 \caption{The three-disk reception filter in normalized first-bearing coordinates. The sector half-angle is enlarged to $12^\circ$ for visibility; the numerical experiment uses $1^\circ$. The shaded set is the intersection of the three blue disks.}
 \label{fig:safe}
\end{figure}

\section{Conditional Time Optimization}
\subsection{Success Requirement and Time Objective}
Let $K$ index the final attempted reading, including fallback attempts, and let $C_K$ be the terminal service point. We measure time from the end of the first bearing reading:
\begin{equation}
 \begin{aligned}
 L_\pi&=\sum_{i=1}^{K-1}\norm{A_{i+1}-A_i}+\norm{C_K-A_K},\\
 T_\pi&=L_\pi/v_{\mathrm{m}}+\tau_m(K-1)+\tau_t.
 \end{aligned}
 \label{eq:time}
\end{equation}
Here $\tau_m$ is the reading duration and $\tau_t$ the terminal inspection and removal duration. Close-range completion has $C_K=A_K$. All travel and attempted readings, including silent fallback readings, are charged. Computation time is excluded.

We specify an area-uniform distribution on $\mathcal F_1$ rather than infer a posterior. Sampling uses uniform $\omega\in[-\alpha,\alpha]$ and uniform $r^2\in[r_{\mathrm{near}}^2,R_{\max}^2]$, rejecting positions outside $\mathcal D$. Future errors follow Section II. Normal sensing guarantees reception for every compatible $R$; fallback uses only close-range feedback. Thus execution time needs no distribution for $R$.

Let $\xi\in\Xi$ comprise a source, compatible radius, and bounded error sequence consistent with the first reading. Define $S_\pi(q;\xi)=1$ for finite completion with reliable terminal service within $\rho$ of the source, and zero otherwise. With fixed continuation $\pi$, success is a hard constraint and expected time ranks feasible points:
\begin{equation}
 \begin{aligned}
 &\min_{q\in\mathcal C_{\mathrm{safe}}} f_\pi(q),\quad
 f_\pi(q)=\mathbb E[T_\pi(q;\xi)],\\
 &\text{subject to }S_\pi(q;\xi)=1\quad\text{for every }\xi\in\Xi.
 \end{aligned}
 \label{eq:objective}
\end{equation}
The expectation uses the assumed distributions; success covers every compatible scenario by Proposition 2. Every simulated mission and its full cost are retained. Only the second point is optimized.

\subsection{Monte Carlo Evaluation and Search}
The sensor moves to $q$ and takes a reading. A close-range flag permits service there; otherwise the bearing updates $\Omega_n$. If \eqref{eq:stop} holds, it moves to $C_n$ for service; otherwise it measures at another admissible point. No remaining candidate or reaching the reading limit triggers fallback. Equation~\eqref{eq:time} counts all moves, readings, and terminal service.

The expectation in \eqref{eq:objective} averages complete mission times over source positions and future errors. Monte Carlo simulation generates $M$ missions from the assumed distributions \cite{owen}; their mean completion time is the sample-average approximation \cite{saa}:
\begin{equation}
 \widehat f_M(q)=\frac1M\sum_{j=1}^M T_\pi(q;G^{(j)},\epsilon^{(j)}).
 \label{eq:saa}
\end{equation}
Each $j$ is a complete mission, including further readings and fallback when needed. Simulation estimates time; geometry establishes success.

Candidates share source and error draws for paired comparisons \cite{crn}. We filter a $20$ m grid on $0\leq u\leq1000$, $-880\leq v\leq880$ by \eqref{eq:filter}, then refine a $\pm20$ m neighborhood of the best point and its reflection using a $2$ m grid. The reference case has 3021 coarse points and 882 refinement records, totaling 3885 distinct locations. Selection minimizes \eqref{eq:saa} on this finite set, without a continuous optimality guarantee.

We validate the fixed point on $N$ independent missions, reporting $\bar T=N^{-1}\sum_jT_j$ and the approximate 95\% confidence interval $\bar T\pm1.96s_T/\sqrt N$, where $s_T$ is the sample standard deviation \cite{owen}. For paired designs, the formula applies to within-mission time differences. These intervals quantify sampling error, not model uncertainty.

\subsection{Continuation Policy}
The policy maintains $\Omega_n$ and a bounded planning polygon $\widehat{\mathcal F}_n$ containing all feasible sources. Initially, bearing half-planes intersect outer tangent approximations of $\mathcal D$ and the first reception disk, with normals spaced by $0.25^\circ$; the close-range exclusion is relaxed to $s_1\geq r_{\mathrm{near}}\cos\alpha$. New wedges clip this polygon, while stopping is checked on $\Omega_n$.

If further sensing is needed, let $\widehat C_n$ be the enclosing-circle center of the planning polygon. In first-bearing coordinates, generate six locations
\begin{equation}
 q=\widehat C_n\pm b n_1,\qquad b\in\{40,80,160\}\ \mathrm{m}.
 \label{eq:continuation}
\end{equation}
Exclude observed points. A candidate is admissible if it lies in a three-disk safe region from any previous normal reading, or if all planning-polygon vertices lie within $R_{\min}$ of it. Convexity extends the latter test to the whole polygon; either test guarantees reception.

Rank candidates using three representative sources: $\widehat C_n$ and its midpoints with the minimum- and maximum-$u$ polygon vertices. Append their hypothetical noiseless bearings to the observed constraints. Let $\widetilde r(q)$ be the largest resulting enclosing-circle radius, and define the estimated travel time $\ell(q)=(\norm{A_n-q}+\norm{q-\widehat C_n})/v_{\mathrm{m}}$. In the reported implementation, the score is
\begin{equation}
 J(q)=\begin{cases}
 \ell(q),&\widetilde r(q)\leq\rho-\delta,\\
 10^6+\widetilde r(q)+10^{-4}\ell(q),&\text{otherwise}.
 \end{cases}
 \label{eq:heuristic}
\end{equation}
The score uses meters and seconds. Its three scenarios guide ranking, not worst-case certification. Actual readings update $\Omega_n$; service still requires \eqref{eq:stop}.

All cases share the same offsets, representative sources, and reading limit. After eight readings following the first, or when candidates run out, fallback scans a $3$ m square lattice over the planning rectangle in serpentine order. Coordinate bounds are rounded outward and padded by one lattice spacing. Only the close-range flag affects this scan; all other feedback is ignored and the planning region is unchanged. Every attempt incurs full cost; incomplete simulations are errors, never omitted.

\subsection{Completion Guarantee and Evaluation Metrics}
\begin{proposition}
Assume exact geometry, bounded bearing errors, accurate unobstructed motion with $v_{\mathrm{m}}>0$, reliable feedback and terminal service, and $0<r_{\mathrm{near}}\leq\rho$. If the planning polygon contains the true source and the fallback lattice spacing $h$ satisfies $h\sqrt2/2\leq r_{\mathrm{near}}$, the specified policy satisfies $S_\pi(q;\xi)=1$ for every $q\in\mathcal C_{\mathrm{safe}}$ and $\xi\in\Xi$.
\end{proposition}
\begin{IEEEproof}
The reception and admissibility tests guarantee usable feedback during normal sensing. Every normal wedge contains $G$, so clipping preserves it in both the angular region and the bounded planning polygon. A stop by \eqref{eq:stop} is successful by \eqref{eq:certificate}; a close-range stop is successful since $r_{\mathrm{near}}\leq\rho$. Otherwise normal sensing ends after at most eight additional readings. The planning rectangle is bounded and contains $G$. Rounding and padding retain a nearest lattice point to $G$, so the finite scan reaches within $h\sqrt2/2$ of $G$, triggering close-range completion. All moves and readings are finite, hence so is completion time.
\end{IEEEproof}

We report $\widehat p_{\mathrm{succ}}=N^{-1}\sum_jS_\pi(q;\xi_j)$ and $\widehat p_2=N^{-1}\sum_j\mathbf1\{K_j=2\}$. The latter means no third reading, though terminal travel may remain. All branches contribute to mean time. Algorithm~\ref{alg:policy} summarizes the method.

\begin{figure}[t]
\begin{minipage}{\columnwidth}
\refstepcounter{algorithm}\label{alg:policy}
\hrule\vspace{3pt}
\textbf{Algorithm \thealgorithm: Reception-Constrained Selection and Completion}
\vspace{2pt}\hrule\vspace{3pt}
\begin{algorithmic}[1]
\REQUIRE First reading, source domain, sensing parameters, and candidate grids.
\STATE Construct $\Omega_1$, $\widehat{\mathcal F}_1$, and $\mathcal C_{\mathrm{safe}}$.
\STATE Fix continuation $\pi$; evaluate every feasible grid point on the same $M$ complete scenarios.
\STATE Select the candidate with the lowest sample mean time; validate on independent scenarios.
\STATE Move to the selected second point and obtain feedback.
\IF{a close-range flag is received}
 \STATE Execute terminal service at the current point.
\ELSE
 \STATE Update both regions from the real bearing; compute the enclosing circle of $\Omega_n$.
 \STATE If $r_n\leq\rho-\delta$, move to $C_n$ and execute service.
 \STATE Otherwise choose an admissible point using \eqref{eq:heuristic}; continue sensing, allowing at most eight readings after the first.
 \STATE If no candidate remains or the limit is reached, scan the padded rectangle on the finite lattice until close-range completion.
\ENDIF
\end{algorithmic}
\vspace{3pt}\hrule
\end{minipage}
\end{figure}

\section{Numerical Study}
\subsection{Benchmark and Reproducibility}
Table~\ref{tab:parameters} gives the parameters. The reference uses $A_1=(1500,0)$ m and bearing $0.4^\circ$. Training and independent validation each use 5000 missions (seeds 20260912 and 20260913). These single-source simulations use the area-uniform conditional distribution; they are not official multi-source tests.

\begin{table}[t]
\caption{Simulation Parameters}
\label{tab:parameters}
\centering\small\setlength{\tabcolsep}{4pt}
\begin{tabular}{llll}
\toprule
Parameter & Value & Parameter & Value\\
\midrule
$a$ & 1800 m & $r_{\mathrm{near}}$ & 5 m\\
$R_{\min}$ & 1000 m & $\rho$ & 20 m\\
$R_{\max}$ & 1500 m & $\delta$ & 0.01 m\\
$\alpha$ & $1^\circ$ & $v_{\mathrm{m}}$ & 5 m/s\\
$\tau_m$ & 5 s & $\tau_t$ & 5 s\\
Training $M$ & 5000 & Validation $N$ & 5000\\
Coarse grid & 20 m & Fine grid & 2 m\\
\bottomrule
\end{tabular}
\end{table}

The package supplies code, samples, and candidate values, reproducing the search and validation mean. Checks verify \eqref{eq:time}, compare the reception filter against dense sector constraints at 2000 random points, and confirm the acute-triangle counterexample. Proposition 1, rather than these sampled checks, establishes reception.

\begin{figure}[t]
 \centering\includegraphics[width=0.92\columnwidth]{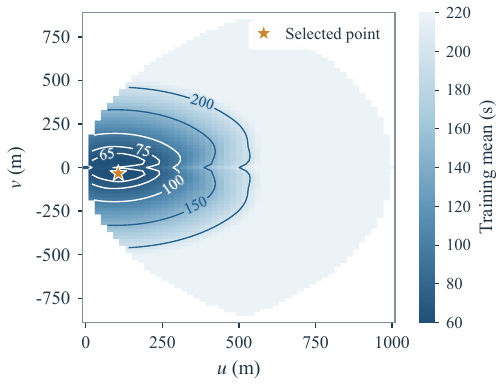}
 \caption{Coarse-grid sample-average completion time for 5000 training missions. The color scale is capped at 220 s; labeled contours interpolate grid values without smoothing. The star marks the refined selected point; validation uses independent missions.}
 \label{fig:objective}
\end{figure}

\begin{figure}[t]
 \centering\includegraphics[width=0.98\columnwidth]{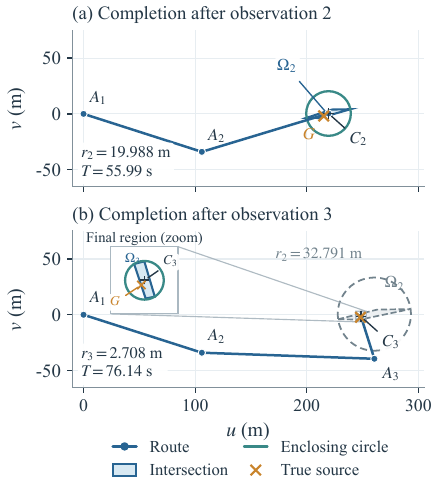}
 \caption{Held-out trajectories and wedge intersections. (a) $r_2=19.988$ m permits service at $C_2$. (b) $r_2=32.791$ m requires $A_3$, giving $r_3=2.708$ m. Blue shading shows the final intersection; the dashed gray outline shows the earlier region. The inset enlarges $\Omega_3$.}
 \label{fig:routes}
\end{figure}

\begin{figure*}[t]
 \centering\includegraphics[width=0.94\textwidth]{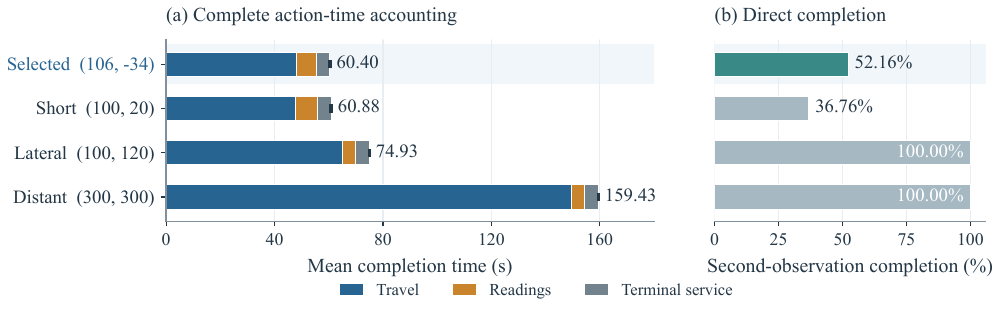}
 \caption{Reference-case validation on the same 5000 missions. (a) Mean travel, reading, and terminal-service costs; black caps give approximate 95\% intervals for total mean time. Coordinates are in meters. (b) Proportion completed without a third reading. Final empirical success is 100\% for every choice, including missions requiring further readings.}
 \label{fig:accounting}
\end{figure*}

\begin{figure}[t]
 \centering\includegraphics[width=0.92\columnwidth]{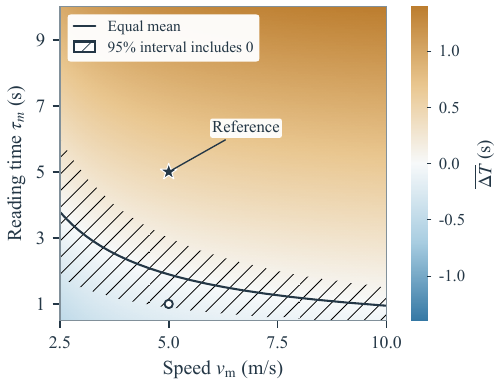}
 \caption{Fixed-trajectory cost sensitivity: positive $\overline{\Delta T}$ favors the selected point over the short-travel design. The zero contour separates sample-mean rankings; hatching denotes pointwise 95\% paired intervals containing zero. The star marks $(v_{\mathrm{m}},\tau_m)=(5,5)$ and the open circle $(5,1)$. All 5000 missions and sensing decisions are fixed.}
 \label{fig:cost}
\end{figure}

\begin{figure}[t]
 \centering\includegraphics[width=0.92\columnwidth]{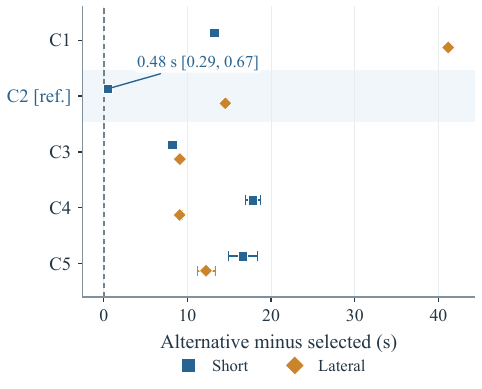}
 \caption{Paired completion-time differences in the five geometries of Table~\ref{tab:geometries}, each with 5000 training and 5000 independent validation missions. Positive alternative-minus-selected differences favor selection; bars give approximate pointwise 95\% intervals. All choices share the same continuation rule.}
 \label{fig:geometry}
\end{figure}

\subsection{Selected Point and Complete Time Accounting}
Selection gives $(106,-34)$ m locally, or $A_2^*=(1606.235,-33.259)$ m globally, with an initial move of 111.32 m. Training and validation means are 60.16 and 60.40 s; validation standard error is 0.223 s and its 95\% interval is [59.96, 60.84] s. Fig.~\ref{fig:objective} shows the coarse objective and Supplementary Fig.~S1 the refinement.

All 5000 missions succeed without fallback: $\widehat p_{\mathrm{succ}}=100\%$, with maximum terminal distance 17.78 m below $\rho=20$ m. Of these, 52.16\% need no third reading; 47.84\% need one more. Mean cost is 48.007 s of travel (240.04 m), 7.392 s of readings (1.4784 attempts), and 5 s of service. Fig.~\ref{fig:routes} shows both branches. True source positions generate feedback and verify distance; the policy uses only observed feedback.

\subsection{Paired Comparisons and the Travel--Sensing Tradeoff}
Fig.~\ref{fig:accounting} compares the selected point with three prescribed alternatives under identical reception and continuation rules, not reproductions of published algorithms. All four complete every validation mission. The lateral point $(100,120)$ m needs no third reading but averages 74.93 s: selection saves 14.53 s (19.39\%; paired 95\% interval [14.28, 14.78] s). The distant point $(300,300)$ m averages 159.43 s.

The short point $(100,20)$ m averages 60.88 s. Selection saves 0.478 s (0.79\%; paired 95\% interval [0.289, 0.668] s), adding 0.292 s of travel but saving 0.770 s of readings. Direct completion rises from 36.76\% to 52.16\%. The modest advantage depends on the cost of further sensing.

\subsection{Monte Carlo Precision and Cost Sensitivity}
Validation-prefix means and intervals are supplied in the reproducibility package. These overlapping samples illustrate precision, not optimization convergence.

We hold the simulated trajectories and sensing decisions fixed and recalculate \eqref{eq:time} for alternative speeds and reading times. Let $\Delta T$ denote the short point's time minus the selected point's time. The saved paired missions give the sample-mean difference
\begin{equation}
 \overline{\Delta T}=-1.4589/v_{\mathrm{m}}+0.1540\tau_m.
 \label{eq:sensitivity}
\end{equation}
Fig.~\ref{fig:cost} covers $2.5\leq v_{\mathrm{m}}\leq10$ m/s and $0.5\leq\tau_m\leq10$ s. Selection has lower sample mean above $v_{\mathrm{m}}\tau_m=9.473$ m. Hatching marks settings whose pointwise paired 95\% interval includes zero. At $v_{\mathrm{m}}=5$ m/s and $\tau_m=1$ s, $\overline{\Delta T}=-0.138$ s favors the short point, but the paired interval for $\Delta T$, [$-0.292$, 0.017] s, includes zero. Supplementary Fig.~S2 shows the fixed-speed slice. Routes and decisions are unchanged; neither selection nor continuation is reoptimized.

\subsection{Additional Initial Geometries}
We repeat selection in four additional geometries with unchanged parameters, grids, and continuation (Table~\ref{tab:geometries}). Each case uses 5000 training and 5000 independent validation missions. Along the first bearing, feasible sources extend to approximately 100, 300, 800, 995, and 1500 m from $A_1$ in C1--C5.

\begin{table}[t]
\caption{Initial Geometries and Independent Validation}
\label{tab:geometries}
\centering\small\setlength{\tabcolsep}{2.7pt}
\begin{tabular}{lrrrrr}
\toprule
Case & $r$, m & $\theta_1$, deg & $(u^*,v^*)$, m & $\bar T$, s & $\widehat p_2$, \%\\
\midrule
C1 & 1700 & 0 & (34,12) & 25.33 & 98.40\\
C2 & 1500 & 0.4 & (106,-34) & 60.40 & 52.16\\
C3 & 1000 & 0 & (212,54) & 133.74 & 18.90\\
C4 & 1500 & 90 & (292,74) & 162.85 & 18.20\\
C5 & 1500 & 180 & (436,-116) & 245.27 & 13.38\\
\bottomrule
\end{tabular}
\par\vspace{2pt}\footnotesize $A_1=(r,0)$ m; bearings use the global $+x$ axis. Selected points use first-bearing coordinates. Each row has 5000 successful validation missions; $\widehat p_2$ is direct completion, not final success.
\end{table}

Fig.~\ref{fig:geometry} compares each reselected point with the fixed short and lateral designs. All paired 95\% intervals lie above zero. The smallest gain is 0.478 s against the short point in C2; lateral-point gains range from 9.08 to 41.18 s. Supplementary Fig.~S3 reports absolute means. These results apply to the specified geometries and designs.

All 25,000 selected-point validation missions succeed without fallback, with maximum terminal distance 19.20 m. Direct completion ranges from 13.38\% to 98.40\%; some C3--C5 missions need a fourth reading. In C5, 36,442 of 19,425,000 training candidate--mission evaluations use fallback, with full costs retained. This is not the selected policy's fallback probability.

Independent 100,000-mission checks in each of C2 and C5 give short-minus-selected differences of 0.502 s [0.459, 0.544] and 16.295 s [15.920, 16.670], respectively (paired 95\% intervals). Both policies complete every mission without fallback: 400,000 executions in total. No observed fallback does not imply zero probability.

Full search and validation took 79--186 s per geometry with eight CPU threads. These offline costs are excluded from $T_\pi$; real-time decision latency remains untested.

\subsection{Selection-Objective Ablation}
The ablation uses the same 3021 coarse points, 5000 training missions, and continuation for three objectives: mean action time, total travel, and additional readings. Reading-count ties use lower travel; travel ties use fewer readings; remaining ties use lexicographic $(u,v)$ order. All three exclude refinement and use the same 5000 held-out missions per geometry.

\begin{table}[t]
\caption{Objective Ablation: Validation Mean Time (s)}
\label{tab:ablation}
\centering\small
\begin{tabular}{lrrr}
\toprule
Case & Time objective & Travel objective & Reading objective\\
\midrule
C1 & 26.40 & 26.40 & 26.40\\
C2 & 60.55 & 60.90 & 66.33\\
C3 & 133.78 & 133.78 & 281.45\\
C4 & 162.92 & 162.92 & 285.90\\
C5 & 245.32 & 245.63 & 310.92\\
\bottomrule
\end{tabular}
\end{table}

Time and travel objectives select the same point in C1, C3, and C4 (Table~\ref{tab:ablation}). In C2 and C5, travel-only selection adds 0.348 s [0.141, 0.554] and 0.319 s [0.102, 0.535], respectively (paired 95\% intervals). Reading-only selection reduces readings but raises mean time in C2--C5, with positive paired intervals. Every choice completes all missions. This isolates the objective on a common finite grid, not superiority over published methods.

\section{Discussion and Conclusion}
Reception, enclosing-circle coverage, and finite fallback guarantee completion under the model. Selected points reduce mean time against two fixed designs in five geometries. The ablation isolates reading costs; cheaper readings can reverse the short-versus-selected sample-mean ranking, though its paired interval includes zero.

Uniform conditional sources and bounded independent errors are ideal assumptions. Navigation error, obstacles, multipath outliers, source motion, and directional reception remain untested. Finite grids and heuristic continuation restrict optimization. Approximate intervals may miss rare fallback costs; excluded computation prevents claims of lower end-to-end time. Published-method comparisons, other priors, and field tests remain necessary. Future work should exploit domain constraints and optimize selection jointly with continuation.


\end{document}